\documentclass[a4paper,reqno]{amsart}
\usepackage{geometry}
\usepackage[T1]{fontenc}
\usepackage[utf8]{inputenc}
\usepackage{lmodern}
\usepackage{amsmath,amssymb,amsthm,mathtools}
\usepackage{setspace}
\usepackage{microtype}
\usepackage{enumitem,booktabs,needspace}
\usepackage{xcolor}
\usepackage[colorlinks]{hyperref}
\usepackage[nameinlink,capitalize,noabbrev]{cleveref}
\hypersetup{
 pdftitle={Relative Suspension and Higher Delooping for Support Tau-Tilting Modules},
 pdfauthor={Mingfei Xu and Xiaojin Zhang},
 pdfsubject={Relative homological algebra and support tau-tilting theory},
 pdfkeywords={relative suspension, support tau-tilting, relative transpose, higher delooping, tensor induction, finitistic dimension},
 linkcolor=blue,
 citecolor=blue,
 urlcolor=blue
}
\numberwithin{equation}{section}
\setlist[enumerate]{label=\textup{(\arabic*)},itemsep=3pt,topsep=5pt}

\newtheorem{introtheorem}{Theorem}

\newtheorem{theorem}{Theorem}[section]
\newtheorem{proposition}[theorem]{Proposition}
\newtheorem{lemma}[theorem]{Lemma}
\newtheorem{corollary}[theorem]{Corollary}
\theoremstyle{definition}
\newtheorem{definition}[theorem]{Definition}
\newtheorem{remark}[theorem]{Remark}
\newtheorem{example}[theorem]{Example}
\crefname{introtheorem}{Theorem}{Theorems}
\crefname{theorem}{Theorem}{Theorems}
\crefname{proposition}{Proposition}{Propositions}
\crefname{lemma}{Lemma}{Lemmas}
\crefname{corollary}{Corollary}{Corollaries}
\crefname{definition}{Definition}{Definitions}
\crefname{remark}{Remark}{Remarks}
\crefname{example}{Example}{Examples}
\Crefname{introtheorem}{Theorem}{Theorems}
\Crefname{theorem}{Theorem}{Theorems}
\Crefname{proposition}{Proposition}{Propositions}
\Crefname{lemma}{Lemma}{Lemmas}
\Crefname{corollary}{Corollary}{Corollaries}
\Crefname{definition}{Definition}{Definitions}
\Crefname{remark}{Remark}{Remarks}
\Crefname{example}{Example}{Examples}

\DeclareMathOperator{\add}{add}
\DeclareMathOperator{\Fac}{Fac}
\DeclareMathOperator{\Sub}{Sub}
\DeclareMathOperator{\End}{End}
\DeclareMathOperator{\Hom}{Hom}
\DeclareMathOperator{\Ext}{Ext}
\DeclareMathOperator{\Tor}{Tor}
\DeclareMathOperator{\Tr}{Tr}
\DeclareMathOperator{\dell}{dell}
\DeclareMathOperator{\ddell}{ddell}
\DeclareMathOperator{\findim}{findim}
\DeclareMathOperator{\Findim}{Findim}
\DeclareMathOperator{\gldim}{gl.dim}
\DeclareMathOperator{\modu}{mod}
\DeclareMathOperator{\proj}{proj}
\DeclareMathOperator{\pd}{pd}
\newcommand{\op}{\mathrm{op}}
\DeclareMathOperator{\ann}{ann}
\DeclareMathOperator{\rad}{rad}
\newcommand{\C}{\mathcal C}
\newcommand{\Y}{\mathcal Y}
\newcommand{\D}{\mathbb D}
\newcommand{\Abar}{\overline A}

\newcommand{\stC}{\underline{\mathcal C}_T}
\newcommand{\uHom}{\underline{\Hom}}

\title[Relative suspension and higher delooping]
{Relative Suspension and Higher Delooping for Support $\tau$-Tilting Modules}
\author[Mingfei Xu and Xiaojin Zhang]{Mingfei Xu and Xiaojin Zhang}
\thanks{Corresponding author: Xiaojin Zhang.}
\makeatletter
\@namedef{subjclassname@2020}{\textup{2020} Mathematics Subject Classification}
\makeatother
\date{}

\begin{document}
\begin{abstract}
We construct relative suspension for the exact category generated by a support $\tau$-tilting module and characterize higher delooping by the units of the resulting adjunction, which are the generalizations of G\'elinas's results in Adv. Math.(2022). We prove that comparison with ordinary higher delooping requires a shift of one in the index. More precisely, for every $k,d\ge1$, there is a module whose ordinary $k$-delooping level is zero and whose relative level is $d$. We also prove that tensor induction preserves the relative constructions and higher delooping levels on induced modules. As an application, we study cyclic Nakayama algebras with arbitrary local coefficients and automorphism twists. Their left and right little and big finitistic dimensions and ordinary and derived delooping levels all equal the finitistic dimension of the underlying Nakayama algebra.
\end{abstract}
\maketitle

\section{Introduction}\label{sec:introduction}

Let $A$ be a finite-dimensional basic algebra over an algebraically closed field $K$, and let $T$ be a nonzero basic support $\tau$-tilting right $A$-module. All modules are finitely generated right modules unless a left action is displayed. Throughout the paper, put
\begin{equation}\label{eq:setup}
 I=\ann_A(T),\qquad \Abar=A/I,\qquad B=\End_A(T),\qquad \C=\Fac T.
\end{equation}
Here $\add T$ consists of the direct summands of finite direct sums of copies of $T$, and $\Fac T$ consists of the factor modules of such sums. The module $T$ is $1$-tilting over $\Abar$ \cite{AIR}. Consequently, $\C$ is an exact category with projectives $\add T$. We write $\Omega_T$ for its syzygy functor and $\stC=\C/[\add T]$ for its stable category. The Brenner--Butler equivalence identifies $\C$ with the torsionfree class $\Sub\D T$ in $\modu B$, where $\D=\Hom_K(-,K)$. We study relative suspension in this category and its relation to ordinary delooping over $B$.

G\'elinas introduced delooping level as a bound for finitistic dimension and studied it through the Auslander--Reiten adjunction $\Sigma=\Tr\Omega\Tr\dashv\Omega$ \cite{AuslanderReiten96,Gelinas}. Guo--Igusa introduced higher and derived delooping levels \cite{GuoIgusa}; their symmetry and their relation to other homological dimensions are studied in \cite{BarriosLanzilottaMata,GuoSymmetry}. Chen--Li--Zhang--Zhao used delooping in the exact category generated by a tilting module \cite{ChenLiZhangZhao}. Xu--Zhang defined relative higher delooping and obtained a finitistic-dimension bound for the endomorphism algebra \cite{XuZhang}. It is natural to ask: Is there a relative version of the Auslander--Reiten adjunction for the relative delooping level? In this paper, we give a positive answer to the question and generalize G\'elinas's results systematically.

For $M\in\C$ and $k\ge1$, the relative level $k$-$\dell_T(M)$ is the least $d\ge0$ such that $\Omega_T^dM$ is a retract of $\Omega_T^{d+k}N$ in $\stC$ for some $N\in\C$. Its value is $\infty$ if no such $d$ exists. The specialization $T=A$ gives ordinary higher delooping. Set
\[
 F=\Hom_A(T,-),\qquad G=-\otimes_BT.
\]
Our first result constructs a left adjoint of relative syzygy and identifies a canonical delooping witness.

\begin{introtheorem}[see \cref{thm:adjunction,thm:unit}]\label{intro:suspension}
The functor $\Omega_T$ on $\stC$ has a left adjoint
\[
 \Sigma_T=\underline G\,\Sigma_B\,\underline F,
 \qquad \Sigma_B=\Tr_{B^{\op}}\Omega_{B^{\op}}\Tr_B.
\]
If $a_X:X\to T^X$ is a left $\add T$-approximation, then $\Sigma_TX$ is represented by $\operatorname{Coker}a_X$. For $M\in\C$, $d\ge0$ and $k\ge1$, one has $k$-$\dell_T(M)\le d$ if and only if the iterated unit
\[
 \eta_{d+k,\Omega_T^dM}:\Omega_T^dM\longrightarrow
 \Omega_T^{d+k}\Sigma_T^{d+k}\Omega_T^dM
\]
is split monic in $\stC$.
\end{introtheorem}

The construction combines the stable tensor--Hom adjunction with the ordinary suspension over $B$. Full faithfulness of $\underline F$ identifies the right adjoint, and the unit characterization then follows formally. Relative torsionfreeness, in Huang's approximation framework and the Xi-transpose formulation \cite{HuangTranspose,HuangOmega,HuangSyzygy,HuangExtension,Xi,ZhangGeng}, gives a sufficient condition: if $\Sigma_T^n\Omega_T^nM$ is $k$-$T$-torsionfree, then $k$-$\dell_T(M)\le n$ (\cref{thm:higher}). Combined with the finite tests in \cite{XuZhang}, this yields \cref{cor:findim} on finitistic dimension.

The equivalence $F:\C\simeq\Sub\D T$ identifies syzygies, but a witness for ordinary delooping may lie outside $\Sub\D T$. \cref{prop:comparison} gives
\begin{equation}\label{eq:intro-comparison}
 k\text{-}\dell_BF(M)\le k\text{-}\dell_T(M)
 \le(k+1)\text{-}\dell_BF(M).
\end{equation}
The first inequality is from \cite{XuZhang}; the second uses $\Omega_B(\modu B)\subseteq\Sub\D T$. \cref{prop:comparison-equality} gives equality on the left when the ordinary canonical witness has vanishing first Tor with $T$. The following result shows that the shift on the right is essential at every index.

\begin{introtheorem}[see \cref{cor:all-index-gaps}]\label{intro:gaps}
For every pair of integers $k,d\ge1$, there exist a finite-dimensional algebra $A$, a classical tilting $A$-module $T$ with $\pd_AT=1$, and $M\in\Fac T$ such that, for $B=\End_A(T)$ and $F=\Hom_A(T,-)$,
\[
 k\text{-}\dell_BF(M)=0<d=k\text{-}\dell_T(M)
 =(k+1)\text{-}\dell_BF(M).
\]
Every simple module on either side of either algebra has infinite projective dimension, whereas
\[
 \findim A=\findim A^{\op}=d+k-1,
 \qquad \findim B=\findim B^{\op}=d+k.
\]
\end{introtheorem}

We first construct a tilting module whose endomorphism algebra is a radical-square-zero algebra on a linearly oriented quiver. Its complete higher-delooping profile records the effect of omitting one ordinary witness. Tensor induction then produces the examples of infinite global dimension. The required compatibility is proved in \cref{thm:tensor-induction}: induction preserves relative syzygy, suspension and transpose, and preserves and reflects higher delooping levels and relative torsionfreeness on induced modules. The support $\tau$-tilting construction is due to Li--Zhang \cite[Theorem~3.7]{LiZhangTensor}. For ordinary delooping, the induction inequality is in \cite[Lemma~4.1]{GuoSymmetry}; restriction provides its converse here.

Our final application concerns cyclic Nakayama algebras with local coefficients and automorphism twists. Write $\findim$ and $\Findim$ for the little and big right finitistic dimensions, and $\dell$ and $\ddell$ for ordinary and derived delooping, defined using simple modules. The coefficient extension in the next statement is defined in \cref{subsec:cyclic-uniform}.

\Needspace{8\baselineskip}
\begin{introtheorem}[see \cref{thm:cyclic-uniform} and \cref{cor:cyclic-big-derived}]\label{intro:cyclic}
Let $N(c)$ be a cyclic Nakayama algebra with at least two simple modules, let $R$ be a finite-dimensional local $K$-algebra with residue field $K$, and let $\boldsymbol\sigma$ be a tuple of coefficient automorphisms. Put $\Lambda=\Lambda_R(c;\boldsymbol\sigma)$. For $\Gamma\in\{\Lambda,\Lambda^{\op}\}$, one has
\[
 \findim\Gamma=\Findim\Gamma=\ddell\Gamma=\dell\Gamma
 =\findim N(c)<\infty.
\]
\end{introtheorem}

No finite-global-dimension assumption on $N(c)$ is needed. The little finitistic-dimension equalities also follow from Xi--Xu's twisted tensor product theorem \cite[Theorem~1.5]{XiXu}. Our proof lifts path modules and uses Ringel's formula \cite{RingelNakayama} to construct finite-projective-dimension modules containing each simple module. This computes both delooping levels and exhibits modules attaining the common value. Two explicit cyclic families have finitistic dimensions $2n-2$ and $2n-3$, respectively. Together they realize every value at least four, although all simple modules have infinite projective dimension and no ordering makes the algebras standardly stratified.

The paper is structured as follows. In \cref{sec:preliminaries}, we recall the exact category and construct relative suspension. In \cref{sec:higher}, we prove the unit characterization, the comparison inequalities and the torsionfreeness criteria. \cref{sec:comparison} proves the finite-test bound, computes the comparison gaps and establishes tensor compatibility. In \cref{sec:cyclic}, we prove \cref{intro:cyclic} and calculate the two cyclic families.

\section{Preliminaries and relative suspension}\label{sec:preliminaries}

In this section, we recall the exact category associated with $T$ and construct the adjunction in \cref{intro:suspension}.

We use $\mathbb N=\{0,1,2,\ldots\}$. Multiplication in $B=\End_A(T)$ is composition. Thus $T$ is a $B$--$\Abar$-bimodule, and $F(X)$ is a right $B$-module by precomposition.

\subsection{The exact category}

To justify the reduction in \eqref{eq:setup}, choose an idempotent $e$ such that $T$ is $\tau$-tilting over $A/AeA$. Apply \cite[Proposition~2.2(c)]{AIR} over this quotient. Its quotient by the annihilator of $T$ is $\Abar$, so $T$ is $1$-tilting over $\Abar$. In particular,
\begin{equation}\label{eq:tilting}
 \pd_{\Abar}T\le1,\quad \Ext^1_{\Abar}(T,T)=0,
 \quad \C=\Fac_{\Abar}T=T^{\perp_1}_{\Abar}.
\end{equation}
There is an exact sequence $0\to\Abar\to T^0\to T^1\to0$ with $T^0,T^1\in\add T$. We equip $\C$ with the exact sequences inherited from $\modu\Abar$. Since all their terms are annihilated by $I$, these are the same sequences when viewed in $\modu A$. This observation concerns sequences with all terms in $\C$; it does not assert equality of arbitrary higher Ext-groups over $A$ and $\Abar$.

\begin{lemma}\label{lem:projectives}
The category $\C$ has enough projectives, with $\proj\C=\add T$. For every $X\in\C$, a right $\add T$-approximation gives a conflation
\begin{equation}\label{eq:relative-cover}
 0\longrightarrow\Omega_T X\longrightarrow T_X
 \xrightarrow{p_X}X\longrightarrow0.
\end{equation}
It remains exact after applying $F$.
\end{lemma}
\begin{proof}
Choose a right $\add T$-approximation $p_X:T_X\to X$. It is surjective because an epimorphism from a finite sum of copies of $T$ to $X$ factors through it. Let $L=\ker p_X$. The approximation property and $\Ext^1_{\Abar}(T,T_X)=0$ give $\Ext^1_{\Abar}(T,L)=0$. Thus $L\in\C$ by \eqref{eq:tilting}, and the sequence is $F$-exact. Objects of $\add T$ are projective in $\C$. Conversely, if $X$ is projective in $\C$, this sequence splits, so $X\in\add T$.
\end{proof}

We write $\uHom_T(X,Y)$ for morphisms modulo those factoring through $\add T$. A stable retract means a retract in $\stC$; equivalently, $X$ is a stable retract of $Y$ if $X$ is a direct summand of $Y\oplus T'$ for some $T'\in\add T$. Indeed, a retraction modulo $\add T$ can be corrected to a retraction through such a direct sum.

\begin{lemma}\label{lem:transport}
The functor $F$ induces an exact equivalence
\begin{equation}\label{eq:Y}
 F:\C\xrightarrow{\sim}\Y,
 \qquad \Y=\{U\in\modu B\mid\Tor_1^B(U,T)=0\}=\Sub\D T,
\end{equation}
with quasi-inverse the restriction of $G$. On stable categories one has
\begin{equation}\label{eq:stable-transport}
 \underline G\dashv\underline F,
 \qquad \underline F\Omega_T\cong\Omega_B\underline F,
\end{equation}
and $\underline F$ is fully faithful.
\end{lemma}
\begin{proof}
The equivalence \eqref{eq:Y} is the Brenner--Butler tilting theorem applied over $\Abar$ \cite{BrennerButler}. It is exact on the indicated subcategories. Notice that $\Y$ contains the projective $B$-modules and is closed under submodules and extensions.

For every $U\in\modu B$, a finite free presentation shows that $G(U)$ is generated by $T$, so $G(U)\in\C$. Also, $F(\add T)=\proj B$ and $G(\proj B)\subseteq\add T$. Hence both functors descend to the stable categories. Under the tensor--Hom adjunction
\[
 \Hom_A(GU,X)\cong\Hom_B(U,FX),
\]
a factorization through $T'\in\add T$ corresponds to one through the projective $F(T')$. Conversely, a factorization through a projective $P$ gives one through $G(P)\in\add T$. This proves the stable adjunction.

Full faithfulness of $F$ identifies factorizations between objects of its image through projectives with factorizations through $\add T$. Thus $\underline F$ is fully faithful. Finally, applying $F$ to \eqref{eq:relative-cover} gives a projective presentation of $F(X)$. The comparison maps for projective presentations yield the natural stable isomorphism in \eqref{eq:stable-transport}.
\end{proof}

Iterating \cref{lem:projectives} gives a proper $\add T$-resolution of each $X\in\C$, meaning that applying $F$ preserves exactness. Minimal right approximations yield a minimal projective resolution of $F(X)$ because $F:\add T\to\proj B$ is an equivalence. In particular, the relative syzygies agree with those used in \cite{XuZhang}.

\subsection{Construction of the left adjoint}

The Auslander--Reiten adjunction on $\underline{\modu}B$ is
\begin{equation}\label{eq:ordinary-adjunction}
 \Sigma_B=\Tr_{B^{\op}}\Omega_{B^{\op}}\Tr_B\dashv\Omega_B;
\end{equation}
see \cite[Corollary~3.4]{AuslanderReiten96} and \cite[Proposition~1.5]{Gelinas}.

\begin{theorem}\label{thm:adjunction}
The endofunctor $\Sigma_T=\underline G\Sigma_B\underline F$ of $\stC$ is left adjoint to $\Omega_T$. If $a_X:X\to T^X$ is a left $\add T$-approximation, then
\begin{equation}\label{eq:cokernel-model}
 \Sigma_T X\cong\operatorname{Coker}a_X\quad\text{in }\stC.
\end{equation}
\end{theorem}
\begin{proof}
\cref{lem:transport} and \eqref{eq:ordinary-adjunction} give natural isomorphisms
\begin{align*}
 \uHom_T(\underline G\Sigma_B\underline F X,Y)
 &\cong\uHom_B(\Sigma_B\underline F X,\underline F Y)\\
 &\cong\uHom_B(\underline F X,\Omega_B\underline F Y)\\
 &\cong\uHom_T(X,\Omega_T Y).
\end{align*}
This proves the adjunction.

For the concrete description, recall that the ordinary suspension of a module $U$ is represented by the cokernel of a left projective approximation of $U$; this follows by dualizing a projective cover of $\Hom_B(U,B)$ in the construction of \eqref{eq:ordinary-adjunction}. Since $F$ is fully faithful on $\C$ and identifies $\add T$ with $\proj B$, the map $F(a_X)$ is a left projective approximation of $F(X)$. Tensor is right exact, and the counits $GF(X)\cong X$ and $GF(T^X)\cong T^X$ identify $G(F(a_X))$ with $a_X$. Therefore $G(\operatorname{Coker}F(a_X))\cong\operatorname{Coker}a_X$, which proves \eqref{eq:cokernel-model}. Its cokernel belongs to $\Fac T$.
\end{proof}

For each $m\ge0$ we have $\Sigma_T^m\dashv\Omega_T^m$. Denote its unit and counit by
\begin{equation}\label{eq:unit-counit}
 \eta_m:\operatorname{Id}\longrightarrow\Omega_T^m\Sigma_T^m,
 \qquad \varepsilon_m:\Sigma_T^m\Omega_T^m\longrightarrow\operatorname{Id}.
\end{equation}
At $m=0$ both transformations are identities. Their components are morphisms in $\stC$.

\begin{remark}
Stable quotients by approximation subcategories occur in the one-sided triangulated framework of Beligiannis--Marmaridis \cite{BeligiannisMarmaridis}. \cref{thm:adjunction} supplies the endomorphism-algebra and cokernel models needed here. It does not assert that $\stC$ is triangulated or that $\Omega_T$ is an equivalence.
\end{remark}

\begin{example}\label{ex:quotient}
Let $e$ be an idempotent with $A/AeA\ne0$ and take $T=A/AeA$, or its basic representative. This is a support $\tau$-tilting module. Then $\C=\modu(A/AeA)$ and the relative constructions recover ordinary syzygy and suspension over $A/AeA$. In particular, taking $T=A$ recovers the classical setting.
\end{example}

\section{Higher delooping and torsionfreeness}\label{sec:higher}

In this section, we prove the unit characterization in \cref{intro:suspension}, compare relative and ordinary higher delooping, and give criteria in terms of relative torsionfreeness.

\subsection{Canonical units and comparison of delooping levels}

\begin{definition}[Xu--Zhang]\label{def:delooping}
For $M\in\C$ and $k\ge1$, define
\begin{equation}\label{eq:relative-delooping}
 k\text{-}\dell_T(M)=\inf\left\{d\ge0\ \middle|\
 \begin{array}{l}
 \Omega_T^dM\text{ is a retract of }\Omega_T^{d+k}N\text{ in }\stC\\
 \text{for some }N\in\C
 \end{array}\right\}.
\end{equation}
The infimum of the empty set is $\infty$. We write $\dell_T(M)$ when $k=1$. Ordinary $k$-$\dell_B$ uses the same definition in $\underline{\modu}B$, with witnesses in $\modu B$.
\end{definition}
This is \cite[Definition~4.4]{XuZhang}, with the nonnegative convention also used for the ordinary invariants \cite{Gelinas,GuoIgusa}.

\begin{theorem}\label{thm:unit}
Let $M\in\C$, $d\ge0$, $k\ge1$ and $X=\Omega_T^dM$. The following are equivalent:
\begin{enumerate}
\item $X$ is a retract of $\Omega_T^{d+k}N$ for some $N\in\C$;
\item $X$ is a retract of $\Omega_T^{d+k}\Sigma_T^{d+k}X$;
\item $\eta_{d+k,X}:X\to\Omega_T^{d+k}\Sigma_T^{d+k}X$ is split monic.
\end{enumerate}
\end{theorem}
\begin{proof}
The implications (3)$\Rightarrow$(2)$\Rightarrow$(1) follow from the definitions. To prove (1)$\Rightarrow$(3), choose $u:X\to\Omega_T^{d+k}N$ and $v:\Omega_T^{d+k}N\to X$ with $vu=1_X$. Let $\widetilde u:\Sigma_T^{d+k}X\to N$ be the morphism adjoint to $u$. Then
\[
 u=\Omega_T^{d+k}(\widetilde u)\eta_{d+k,X}.
\]
Consequently $v\Omega_T^{d+k}(\widetilde u)$ is a left inverse of the unit.
\end{proof}

The argument applies to any adjunction and extends \cite[Theorem~1.10]{Gelinas}. Applying further syzygies to a retraction shows that delooping at level $d$ implies delooping at every higher level. Thus
\begin{equation}\label{eq:unit-invariant}
 k\text{-}\dell_T(M)=
 \inf\{d\ge0\mid\eta_{d+k,\Omega_T^dM}\text{ is split monic}\}.
\end{equation}
When $k$-$\dell_T(M)\le d$, one may choose the witness
\begin{equation}\label{eq:canonical-witness}
 N=\Sigma_T^{d+k}\Omega_T^dM
\end{equation}
up to stable isomorphism. The split retraction itself need not be unique.

\begin{proposition}\label{prop:comparison}
For $M\in\C$ and $k\ge1$,
\begin{equation}\label{eq:comparison}
 k\text{-}\dell_BF(M)\le k\text{-}\dell_T(M)
 \le(k+1)\text{-}\dell_BF(M).
\end{equation}
\end{proposition}
\begin{proof}
Applying $\underline F$ to a relative delooping retraction gives the first inequality by \eqref{eq:stable-transport}.

For the second, suppose that $(k+1)$-$\dell_BF(M)=d<\infty$. Choose a stable retraction of $\Omega_B^dF(M)$ from $\Omega_B^{d+k+1}U$ with $U\in\modu B$. Since $\Y$ contains projectives and is closed under submodules, $\Omega_BU\in\Y$. Choose $N\in\C$ with $F(N)\cong\Omega_BU$. The two terms of this retraction are stably isomorphic to
\[
 F(\Omega_T^dM)\quad\text{and}\quad F(\Omega_T^{d+k}N),
\]
respectively. Full faithfulness of $\underline F$ lifts the retraction to $\stC$. Hence $k$-$\dell_T(M)\le d$. The assertion is immediate if the right-hand side is infinite.
\end{proof}

The first inequality is the comparison used in \cite[Theorem~4.5]{XuZhang}. The second uses the fact that every first $B$-syzygy lies in $\Y$. \cref{thm:unbounded-gap} shows that the first inequality can have an arbitrarily large gap, even when the second is an equality.

\begin{proposition}[An equality criterion]\label{prop:comparison-equality}
Let $M\in\C$, $k\ge1$ and $d\ge0$. Suppose that
$k$-$\dell_BF(M)\le d$, and choose a representative
\[
 W=\Sigma_B^{d+k}\Omega_B^dF(M)
\]
of the ordinary canonical witness. If $\Tor_1^B(W,T)=0$, then
$k$-$\dell_T(M)\le d$. In particular, if
$d=k$-$\dell_BF(M)<\infty$, this condition gives
$k$-$\dell_T(M)=d$.

If $\Sigma_B(\underline\Y)\subseteq\underline\Y$, where
$\underline\Y$ denotes the full stable image of $\Y$, then
\[
 k\text{-}\dell_T(M)=k\text{-}\dell_BF(M)
 \qquad(M\in\C,\ k\ge1),
\]
including infinite values.
\end{proposition}
\begin{proof}
The ordinary version of \cref{thm:unit} makes
$\Omega_B^dF(M)$ a stable retract of $\Omega_B^{d+k}W$.
The Tor condition says precisely that $W\in\Y$.
Writing $W\cong F(N)$ and using full faithfulness of
$\underline F$ lifts this retraction to $\stC$.
\cref{prop:comparison} gives equality when $d$ is the
ordinary level. The Tor condition is independent of the chosen
representative, since positive Tor vanishes on projective summands.

For the last assertion, $\Omega_B^dF(M)$ belongs to $\Y$ because
$\Y$ is closed under syzygies. Closure under $\Sigma_B$ in the
stable category puts the displayed witness in $\underline\Y$,
so its first Tor-group vanishes. Apply the first assertion when
the ordinary level is finite; when it is infinite, use the first
inequality of \cref{prop:comparison}.
\end{proof}

\subsection{Relative transpose and evaluation}

Choose a minimal proper $\add T$-presentation
\begin{equation}\label{eq:proper-presentation}
 T_1\xrightarrow{d}T_0\longrightarrow X\longrightarrow0
\end{equation}
for $X\in\C$, and put
\begin{equation}\label{eq:transpose}
 \Tr_T X=\operatorname{Coker}\bigl(\Hom_A(T_0,T)
 \xrightarrow{\Hom_A(d,T)}\Hom_A(T_1,T)\bigr).
\end{equation}
This is Xi's relative transpose \cite{Xi}, a left $B$-module. Here minimality refers to the successive right $\add T$-approximations. Minimal proper presentations are unique up to isomorphism; arbitrary proper presentations give the same transpose up to projective summands.

\begin{proposition}\label{prop:transpose}
With minimal presentations, $\Tr_T X\cong\Tr_BF(X)$. In particular,
\begin{equation}\label{eq:relative-TrOmegaTr}
 \Sigma_T X\cong
 \bigl(\Tr_{B^{\op}}\Omega_{B^{\op}}\Tr_T X\bigr)\otimes_BT
 \quad\text{in }\stC.
\end{equation}
\end{proposition}
\begin{proof}
Applying $F$ to \eqref{eq:proper-presentation} gives a minimal projective presentation of $F(X)$. The natural identifications
\[
 \Hom_B(F(T_i),B)\cong\Hom_A(T_i,T)
\]
identify the two cokernels defining the transposes. The formula for suspension now follows from \cref{thm:adjunction}. This also proves the comparison in \cite[Corollary~3.6]{Xi} in the present setting.
\end{proof}

For $X\in\C$, write $X^*=\Hom_{\Abar}(X,T)$, a left $B$-module. For a left $B$-module $V$, write $V^\vee=\Hom_{B^{\op}}(V,{}_BT)$, a right $\Abar$-module. Let $\delta_X:X\to(X^*)^\vee$ be evaluation. These conventions fix the sides in all the following Ext-groups.

\begin{lemma}\label{lem:evaluation}
There is an exact sequence of right $\Abar$-modules
\begin{equation}\label{eq:evaluation}
\begin{split}
0\longrightarrow\Ext^1_{B^{\op}}(\Tr_T X,{}_BT)
\longrightarrow X\xrightarrow{\delta_X}(X^*)^\vee
\longrightarrow\Ext^2_{B^{\op}}(\Tr_T X,{}_BT)\longrightarrow0.
\end{split}
\end{equation}
For $q\ge1$ one also has
\begin{equation}\label{eq:double-shift}
 \Ext^{q+2}_{B^{\op}}(\Tr_T X,{}_BT)
 \cong\Ext^q_{B^{\op}}(X^*,{}_BT).
\end{equation}
\end{lemma}
\begin{proof}
Dualizing \eqref{eq:proper-presentation} gives
\[
 0\longrightarrow X^*\longrightarrow T_0^*
 \longrightarrow T_1^*\longrightarrow\Tr_T X\longrightarrow0,
\]
where $T_i^*$ is projective over $B^{\op}$. Split this sequence at the image of $T_0^*\to T_1^*$ and apply $(-)^\vee$ to the two short exact sequences. The evaluation maps identify $(T_i^*)^\vee$ with $T_i$, and the composite $T_1\to T_0$ is $d$. Comparing cokernels gives \eqref{eq:evaluation}, with the middle map equal to evaluation by naturality. This is the evaluation sequence of \cite[Theorem~3.9]{Xi}. Dimension shifting along the displayed sequence gives \eqref{eq:double-shift}.
\end{proof}

\begin{definition}\label{def:torsionfree}
An object $X\in\C$ is \emph{$T$-torsionless} if it embeds into an object of $\add T$. For $k\ge1$, it is \emph{$k$-$T$-torsionfree} if
\begin{equation}\label{eq:torsionfree}
 \Ext^i_{B^{\op}}(\Tr_T X,{}_BT)=0\qquad(1\le i\le k).
\end{equation}
\end{definition}
We use the relative transpose formulation of \cite{ZhangGeng}, with the change of sides required by the right-module convention. The ensuing approximation characterization places these conditions in Huang's framework \cite{HuangOmega}.

\begin{lemma}\label{lem:stable-invariance}
The condition of being $k$-$T$-torsionfree is invariant under isomorphism in $\stC$.
\end{lemma}
\begin{proof}
Stable isomorphisms in $\stC$ give stable isomorphisms under $F$. Ordinary transpose is a duality between the projectively stable categories on the two sides \cite{AuslanderBridger}. Thus \cref{prop:transpose} gives a stable isomorphism of the relative transposes. Positive Ext-groups out of projectives vanish, so \eqref{eq:torsionfree} is unchanged.
\end{proof}

\begin{samepage}
\begin{proposition}\label{prop:torsionless}
For $X\in\C$, the following conditions are equivalent:
\begin{enumerate}
\item $X$ is $T$-torsionless;
\item $\Ext^1_{B^{\op}}(\Tr_T X,{}_BT)=0$;
\item $\uHom_T(X,\D\Abar)=0$.
\end{enumerate}
If these conditions hold, then $X\cong\Omega_TY$ in $\stC$ for some $Y\in\C$.
\end{proposition}
\end{samepage}
\begin{proof}
By \eqref{eq:evaluation}, (2) is equivalent to injectivity of $\delta_X$. A finite free cover of $X^*$ dualizes to an embedding of $(X^*)^\vee$ into a finite sum of copies of $T$. Conversely, an embedding $X\to T'$ with $T'\in\add T$ and naturality of evaluation imply that $\delta_X$ is injective. Hence (1) and (2) are equivalent.

The module $\D\Abar$ is injective over $\Abar$ and belongs to $\C$ by \eqref{eq:tilting}. If $X$ embeds into $T'\in\add T$, every map $X\to\D\Abar$ extends to $T'$, proving (3). Conversely, embed $X$ into an injective module $E\in\add(\D\Abar)$. Under (3), this embedding factors through an object of $\add T$. Its first factor is injective, proving (1).

Finally, the cokernel $Y$ of an embedding $X\to T'$ lies in $\Fac T$. Thus $0\to X\to T'\to Y\to0$ is a conflation with projective middle term, and $X$ represents $\Omega_TY$ in $\stC$.
\end{proof}

\begin{theorem}[Stable Hom criterion]\label{thm:Hom}
For $M\in\C$ and $n\ge0$, the condition
\begin{equation}\label{eq:Hom-condition}
 \uHom_T(\Omega_T^nM,\Omega_T^n\D\Abar)=0
\end{equation}
implies $\dell_T(M)\le n$.
\end{theorem}
\begin{proof}
By adjunction, \eqref{eq:Hom-condition} is equivalent to
\[
 \uHom_T(\Sigma_T^n\Omega_T^nM,\D\Abar)=0.
\]
By \cref{prop:torsionless}, a representative $X$ of $\Sigma_T^n\Omega_T^nM$ satisfies $X\cong\Omega_TY$ for some $Y\in\C$. The triangle identity
\begin{equation}\label{eq:triangle-identity}
 \Omega_T^n(\varepsilon_{n,M})\eta_{n,\Omega_T^nM}
 =1_{\Omega_T^nM}
\end{equation}
makes $\Omega_T^nM$ a retract of $\Omega_T^nX\cong\Omega_T^{n+1}Y$. This proves the claim, including $n=0$.
\end{proof}

For $T=A$, the stable Hom formulation is the one underlying the Artin-algebra criterion in \cite[Corollary~1.15]{Gelinas}. Ordinary Hom-vanishing also suffices, since it implies \eqref{eq:Hom-condition}. For a general support $\tau$-tilting module the appropriate cogenerator is $\D\Abar$, which need not be $\D A$.

\subsection{Higher torsionfreeness}

We next characterize higher relative torsionfreeness by left $\add T$-approximations.

\begin{samepage}
\begin{proposition}[Approximation criterion]\label{prop:approximation}
For $X\in\C$ and $k\ge1$, the following are equivalent:
\begin{enumerate}
\item $X$ is $k$-$T$-torsionfree;
\item there is an exact sequence
\begin{equation}\label{eq:coresolution}
 0\longrightarrow X\longrightarrow T^0\longrightarrow T^1
 \longrightarrow\cdots\longrightarrow T^{k-1},
 \qquad T^i\in\add T,
\end{equation}
such that the following dual sequence is exact:
\begin{equation}\label{eq:dual-coresolution}
 (T^{k-1})^*\longrightarrow\cdots\longrightarrow(T^0)^*
 \longrightarrow X^*\longrightarrow0
\end{equation}
\end{enumerate}
No exactness condition is imposed beyond the last displayed term of \eqref{eq:coresolution}. All its successive cokernels, including the last one, lie in $\C$. In particular, (1) implies $X\cong\Omega_T^kY$ in $\stC$ for some $Y\in\C$.
\end{proposition}
\end{samepage}
\begin{proof}
For $k=1$, condition (1) means that evaluation is injective. Choose a projective epimorphism $Q_0\to X^*$ over $B^{\op}$. Its dual, composed with evaluation, gives an injection $X\to Q_0^\vee$. The identity
\[
 \Hom_{\Abar}(Q_0^\vee,T)\cong Q_0
\]
identifies its dual map with $Q_0\to X^*$. Hence it is a left $\add T$-approximation and gives (2). The converse follows from \cref{prop:torsionless}.

Suppose $k\ge2$. By \cref{lem:evaluation}, condition (1) is equivalent to $\delta_X$ being an isomorphism together with
\begin{equation}\label{eq:dual-vanishing}
 \Ext^q_{B^{\op}}(X^*,{}_BT)=0\qquad(1\le q\le k-2),
\end{equation}
where this range is empty for $k=2$. Dualize a projective resolution $Q_\bullet\to X^*$ by $(-)^\vee$. Evaluation and \eqref{eq:dual-vanishing} give the exact segment
\[
 0\longrightarrow X\longrightarrow Q_0^\vee
 \longrightarrow\cdots\longrightarrow Q_{k-1}^\vee.
\]
Each $Q_i^\vee$ belongs to $\add T$. Applying $(-)^*$ identifies the resulting segment with the original projective resolution, since $Q_i\cong(Q_i^\vee)^*$ naturally. This proves (2).

Conversely, \eqref{eq:dual-coresolution} is a projective resolution segment of $X^*$. Dualizing its first two terms and comparing with \eqref{eq:coresolution} shows that $\delta_X$ is an isomorphism. For $1\le q\le k-2$, the cohomology computing $\Ext^q_{B^{\op}}(X^*,{}_BT)$ is zero by exactness of \eqref{eq:coresolution}. Equations~\eqref{eq:evaluation} and \eqref{eq:double-shift} now give (1).

Let $C_0=X$ and let $C_{i+1}$ be the cokernel of the induced injection $C_i\to T^i$ for $0\le i<k$. Each $C_{i+1}$ is a quotient of $T^i$, so belongs to $\Fac T$. The conflations $0\to C_i\to T^i\to C_{i+1}\to0$ give $C_i\cong\Omega_TC_{i+1}$ stably. Take $Y=C_k$.
\end{proof}

This is the approximation description associated with Huang's relative torsionfreeness and the Xi-transpose formulation of Zhang--Geng \cite{HuangOmega,ZhangGeng}. The proof also explains why an ambient relative syzygy sequence becomes a sequence of conflations in $\C$.

\begin{theorem}[Higher torsionfreeness criterion]\label{thm:higher}
Let $M\in\C$, $n\ge0$ and $k\ge1$. If one, equivalently every, representative of $\Sigma_T^n\Omega_T^nM$ is $k$-$T$-torsionfree, then
\[
 j\text{-}\dell_T(M)\le n\qquad(1\le j\le k).
\]
\end{theorem}
\begin{proof}
Representative independence is \cref{lem:stable-invariance}. By \cref{prop:approximation}, such a representative $X$ is stably isomorphic to $\Omega_T^kY$ for some $Y\in\C$. Equation~\eqref{eq:triangle-identity} makes $\Omega_T^nM$ a retract of $\Omega_T^nX\cong\Omega_T^{n+k}Y$. This proves $k$-$\dell_T(M)\le n$. For $j\le k$, the same argument applies because $k$-$T$-torsionfree implies $j$-$T$-torsionfree.
\end{proof}

\begin{remark}\label{rem:coefficients}
Although $\Tr_T X\cong\Tr_BF(X)$, relative torsionfreeness uses the coefficient module ${}_BT$, whereas ordinary torsionfreeness over $B$ uses ${}_BB$. These conditions should not be interchanged. Similarly, $\Sigma_T$ is obtained by applying $G$ to $\Sigma_BF$, and $G$ can annihilate nonzero $B$-modules. Both distinctions are visible in \cref{ex:strict}.
\end{remark}

\section{Finitistic dimension and comparison gaps}\label{sec:comparison}

In this section, we first express finitistic-dimension bounds through a finite family of relative tests. We then construct the comparison gaps in \cref{intro:gaps}; the passage to infinite global dimension uses tensor induction.

\subsection{A finite family of tests}

For a finite-dimensional algebra $\Lambda$, we write $\findim\Lambda$ for the supremum of the finite projective dimensions of finitely generated right $\Lambda$-modules. We use
\[
 \dell\Lambda=\max\{\dell_\Lambda S\mid
 S\text{ is a simple right }\Lambda\text{-module}\}.
\]
We also write $\Findim\Lambda$ for the supremum of the finite
projective dimensions of all right $\Lambda$-modules, without a
finite-generation assumption, and $\ddell\Lambda$ for the derived
delooping level of Guo--Igusa. We will use their inequality
\cite[Theorem~1.1]{GuoIgusa}
\begin{equation}\label{eq:dimension-chain}
 \findim\Lambda^{\op}\le\Findim\Lambda^{\op}
 \le\ddell\Lambda\le\dell\Lambda.
\end{equation}
The maximum defining $\dell\Lambda$ is taken over simple modules.
The global delooping level, defined using all finitely generated modules
\cite{BarriosLanzilottaMata}, is a different invariant and is not used below.

Write $T=\bigoplus_{i=1}^rT_i$ with $T_i$ indecomposable, and put
\[
 \begin{gathered}
 P_i=F(T_i),\qquad L_i=P_i/\rad P_i,\\
 \mathcal I=\{i\mid L_i\in\Y\},\qquad
 \mathcal J=\{1,\ldots,r\}\setminus\mathcal I.
 \end{gathered}
\]
For $i\in\mathcal I$ set $E_i=G(L_i)$. For $j\in\mathcal J$ set $N_j=G(\rad P_j)$. The module $\rad P_j$ belongs to $\Y$ because $P_j\in\Y$ and $\Y$ is closed under submodules. Thus
\begin{equation}\label{eq:test-objects}
 F(E_i)\cong L_i,\qquad F(N_j)\cong\rad P_j.
\end{equation}
The $E_i$ are pairwise Hom-orthogonal bricks. By \cite[Theorem~4.1]{XuZhang}, they are the bricks of the semibrick associated with $T$. Defining them by \eqref{eq:test-objects} avoids any dependence on a choice of labeling.

\begin{proposition}[Xu--Zhang bound]\label{prop:findim-bound}
With empty index sets omitted, put
\begin{equation}\label{eq:delta}
 \delta_T=\max\left\{
 \sup_{i\in\mathcal I}\dell_T(E_i),\quad
 \sup_{j\in\mathcal J}\bigl(2\text{-}\dell_T(N_j)+1\bigr)
 \right\}.
\end{equation}
Then
\[
 \findim B^{\op}\le\Findim B^{\op}\le\ddell B
 \le\dell B\le\delta_T.
\]
\end{proposition}
\begin{proof}
The inequalities preceding $\dell B$ follow from
\eqref{eq:dimension-chain}. It remains to bound the ordinary
delooping levels of the simple $B$-modules.

For $i\in\mathcal I$, \cref{prop:comparison} gives $\dell_B(L_i)\le\dell_T(E_i)$. For $j\in\mathcal J$, suppose $e=2$-$\dell_T(N_j)<\infty$. Applying $F$ to a witness at level $e$ shows that $\Omega_B^e\rad P_j$ is a retract of $\Omega_B^{e+2}F(Y)$ for some $Y\in\C$. Since $\rad P_j=\Omega_BL_j$, this is a delooping of $L_j$ at level $e+1$. Hence
\[
 \dell_B(L_j)\le2\text{-}\dell_T(N_j)+1.
\]
The inequality is automatic when $e=\infty$. Taking the maximum over the simple modules proves the assertion. This supplies a direct proof of \cite[Theorem~4.5(2)]{XuZhang} with the notation above.
\end{proof}

After relabeling the $E_i$ as the associated semibrick, $\delta_T$ is the invariant $\dell_T\C$ of \cite[Definition~4.4]{XuZhang}. The proof explains both the use of $2$-delooping and the extra $1$ for the indices in $\mathcal J$.

\begin{samepage}
\begin{corollary}\label{cor:findim}
Let $n\ge1$. Assume that
\begin{enumerate}
\item $\Sigma_T^n\Omega_T^nE_i$ has a $1$-$T$-torsionfree representative for every $i\in\mathcal I$;
\item $\Sigma_T^{n-1}\Omega_T^{n-1}N_j$ has a $2$-$T$-torsionfree representative for every $j\in\mathcal J$.
\end{enumerate}
Then $\Findim B^{\op}\le\ddell B\le n$. Condition \textup{(1)} is equivalent to
\begin{equation}\label{eq:finite-Hom}
 \uHom_T(\Omega_T^nE_i,\Omega_T^n\D\Abar)=0
 \qquad(i\in\mathcal I).
\end{equation}
If $\mathcal J$ is empty, the assertion holds for $n\ge0$ using only \textup{(1)}.
\end{corollary}
\end{samepage}
\begin{proof}
\cref{thm:higher} gives $\dell_T(E_i)\le n$ and $2$-$\dell_T(N_j)\le n-1$. Apply \cref{prop:findim-bound}. The equivalence of (1) and \eqref{eq:finite-Hom} follows from \cref{prop:torsionless} and $\Sigma_T^n\dashv\Omega_T^n$. When $\mathcal J$ is empty there is no shifted term, so the same proof allows $n=0$.
\end{proof}

We first record a consequence of the Horseshoe Lemma. For $V\in\C$, let $\pd_TV$ denote the least length of a proper $\add T$-resolution. By \cref{lem:transport}, this equals $\pd_BF(V)$.
The ordinary specialization of the following lemma is the embedding
criterion used in \cite{RingelNakayama}.

\begin{lemma}\label{lem:finite-pd-embedding}
Let $M,V\in\C$ and suppose that there is a monomorphism $M\to V$ with $\pd_TV\le d<\infty$. Then $\dell_T(M)\le d$. In particular, if each simple right module $S$ over a finite-dimensional algebra $\Lambda$ embeds in a module $W_S$ of finite projective dimension, then
\begin{equation}\label{eq:simple-embedding-bound}
 \findim\Lambda^{\op}\le\dell\Lambda
 \le\max_S\pd_\Lambda W_S.
\end{equation}
\end{lemma}
\begin{proof}
Put $Z=V/M$. Since $\Fac T$ is closed under quotients, the sequence $0\to M\to V\to Z\to0$ is a conflation in $\C$. Apply the Horseshoe Lemma to proper $\add T$-resolutions of $M$ and $Z$. At the $d$th syzygy it gives a conflation
\[
 0\longrightarrow\Omega_T^dM\longrightarrow Q
 \longrightarrow\Omega_T^dZ\longrightarrow0,
 \qquad Q\in\add T,
\]
because the $d$th syzygy in any proper resolution of $V$ is projective in $\C$, by the relative Schanuel lemma and $\pd_TV\le d$. Consequently $\Omega_T^dM\cong\Omega_T^{d+1}Z$ in $\stC$, so $\dell_T(M)\le d$. This also covers $d=0$, with $Q=V$. For the last assertion, take $A=T=\Lambda$ and apply the first inequality of \cref{prop:findim-bound}.
\end{proof}

\subsection{Unbounded comparison gaps and sharpness}

Barrios--Lanzilotta--Mata \cite[Example~4.22]{BarriosLanzilottaMata}
give unbounded gaps between an algebra's ordinary delooping level and a
finitistic dimension. Here the comparison is between two higher delooping
levels of corresponding objects: one computed in $\Fac T/[\add T]$ and
the other in $\underline{\modu}B$.

\begin{theorem}[Unbounded relative comparison gaps]\label{thm:unbounded-gap}
For an integer $d\ge1$, let $Q_d$ be the quiver
\[
 0\xrightarrow{\alpha_0}1\xrightarrow{\alpha_1}2
 \longrightarrow\cdots\xrightarrow{\alpha_d}d+1,
\]
and put
\begin{equation}\label{eq:gap-algebra}
 A_d=KQ_d/(\alpha_i\alpha_{i+1}\mid1\le i\le d-1).
\end{equation}
Paths are multiplied in their order of traversal; the relation set
is empty for $d=1$. Write $P_i=e_iA_d$ and $S_i=\operatorname{top}P_i$,
and set
\begin{equation}\label{eq:gap-tilting}
 T_d=S_0\oplus P_0\oplus\bigoplus_{i=2}^{d+1}P_i,
 \qquad X_d=P_0/\rad^2P_0.
\end{equation}
Then $T_d$ is a classical tilting module with $\pd_{A_d}T_d=1$.
For $B_d=\End_{A_d}(T_d)$ and $F_d=\Hom_{A_d}(T_d,-)$, one has
\begin{equation}\label{eq:gap-endomorphism}
 B_d\cong KQ_d/J_{Q_d}^2,\qquad
 \gldim A_d=d,\qquad \gldim B_d=d+1,
\end{equation}
where $J_{Q_d}$ is the arrow ideal. Moreover,
\begin{equation}\label{eq:gap-all-levels}
 \begin{gathered}
 k\text{-}\dell_{T_d}(X_d)=d\qquad(k\ge1),\\
 k\text{-}\dell_{B_d}F_d(X_d)=
 \begin{cases}0,&k=1,\\ d,&k\ge2,\end{cases}
 \end{gathered}
\end{equation}
and the bound in \cref{prop:findim-bound} is sharp:
\begin{equation}\label{eq:gap-delta}
 \delta_{T_d}=d+1=\findim B_d^{\op}.
\end{equation}
\end{theorem}
\begin{proof}
Put $r=d+1$. The lengths of $P_0,\ldots,P_r$ are $3,2,\ldots,2,1$,
respectively. The projective cover of $S_0$ gives
\[
 0\longrightarrow P_1\longrightarrow P_0
 \longrightarrow S_0\longrightarrow0.
\]
Restriction induces a surjection
\[
 \Hom_{A_d}(P_0,P_0)\longrightarrow\Hom_{A_d}(P_1,P_0).
\]
Moreover,
\[
 \Hom_{A_d}(P_1,S_0)=\Hom_{A_d}(P_1,P_i)=0\qquad(i\ge2).
\]
Thus $\Ext^1_{A_d}(S_0,T_d)=0$ and $\pd_{A_d}S_0=1$.
The other summands of $T_d$ are projective, and
\[
 0\longrightarrow A_d\longrightarrow
 P_0^{\oplus2}\oplus\bigoplus_{i=2}^{r}P_i
 \longrightarrow S_0\longrightarrow0
\]
is a tilting coresolution. This proves the tilting assertion.
For $1\le i<r$, one has $\Omega_{A_d}S_i=S_{i+1}$, and
$S_r=P_r$ is projective. Hence $\gldim A_d=r-1=d$.

Order the summands as $U_0=S_0$, $U_1=P_0$, and $U_i=P_i$ for
$2\le i\le r$. Apart from scalar endomorphisms, the only nonzero
maps between these summands are $U_{i+1}\to U_i$ for $0\le i<r$,
each spanning a one-dimensional Hom-space. Every composite of two
successive maps is zero. The endomorphism convention of \cref{sec:preliminaries}
therefore gives $B_d=KQ_d/J_{Q_d}^2$.
Write $Q_i=e_iB_d$ and $L_i=\operatorname{top}Q_i$.
Then $\Omega_{B_d}L_i=L_{i+1}$ for $i<r$, and $L_r=Q_r$.
Consequently $\gldim B_d=r$.

The indecomposable $A_d$-modules are the uniserial quotients of the
$P_i$. Among them, precisely $S_1$ and $P_1$ fail to belong to
$\Fac T_d$: no map from $T_d$ has image meeting their top.
All remaining modules are quotients of summands of $T_d$.
Thus, modulo $\add T_d$, the nonzero indecomposables are
\[
 X_d,\qquad S_i\quad(2\le i<r).
\]
The conflations with middle terms $P_0$ and $P_i$ ($2\le i<r$) give
\begin{equation}\label{eq:gap-relative-chain}
 \Omega_{T_d}X_d=S_2,\qquad
 \Omega_{T_d}S_i=S_{i+1}\quad(2\le i<r),\qquad S_r\in\add T_d.
\end{equation}
Here the displayed epimorphisms are minimal right
$\add T_d$-approximations. The chain is interpreted as terminating
at the relative projective $S_2$ when $d=1$.

For $0\le t<d$, the module $\Omega_{T_d}^tX_d$ is the
nonprojective object in position $t$ of this chain. For any
$N\in\Fac T_d$ and $k\ge1$, each nonprojective summand of
$\Omega_{T_d}^{t+k}N$ occurs strictly later in the chain.
Krull--Schmidt and the description of stable retracts after
\cref{lem:projectives} exclude a retraction onto
$\Omega_{T_d}^tX_d$. At $t=d$ that object is relative projective.
This proves the first equality in \eqref{eq:gap-all-levels}.

Computing Hom from the summands of $T_d$ gives
\[
 F_d(X_d)=L_1,\qquad F_d(S_i)=L_i\quad(2\le i\le r).
\]
Since $L_1=\Omega_{B_d}L_0$, its ordinary delooping level is zero.
For $k\ge2$ and $0\le t<d$, its $t$th syzygy is $L_{t+1}$.
Every indecomposable $B_d$-module is simple or projective, and a
nonprojective $(t+k)$th syzygy of a simple $L_j$ is
$L_{j+t+k}$. Its index is strictly greater than $t+1$.
Thus $L_{t+1}$ cannot be a stable retract of such a syzygy.
At $t=d$, its syzygy $L_r$ is projective, proving the second
equality in \eqref{eq:gap-all-levels}.

The preceding calculations also identify the torsionfree class as
\[
 \Y_d=\add(Q_0,\ldots,Q_{r-1},L_1,\ldots,L_r).
\]
In the notation of \eqref{eq:test-objects},
$\mathcal I=\{1,\ldots,r\}$, $\mathcal J=\{0\}$,
$E_1=X_d$, $E_i=S_i$ for $i\ge2$, and $N_0=X_d$.
The maximum of the unshifted terms in \eqref{eq:delta} is $d$,
whereas the shifted term is $2$-$\dell_{T_d}(X_d)+1=d+1$.
Since $\findim B_d^{\op}=\gldim B_d^{\op}=d+1$,
this proves \eqref{eq:gap-delta}.
\end{proof}

Thus the difference $\dell_T(M)-\dell_BF(M)$ is unbounded,
even for classical tilting modules
of projective dimension one. The shift of the higher-delooping
index in \cref{prop:comparison} is essential.
Likewise, omitting the shifted term in \eqref{eq:delta} gives an
incorrect bound for every member of this family.

\begin{corollary}[The complete higher-level calculation]\label{cor:gap-profile}
In \cref{thm:unbounded-gap}, put $r=d+1$, $E_1=X_d$
and $E_i=S_i$ for $2\le i\le r$. Label the simple $B_d$-modules
so that $F_d(E_i)=L_i$. For $k\ge1$ and $1\le i\le r$,
\begin{equation}\label{eq:relative-gap-profile}
 k\text{-}\dell_{T_d}(E_i)=
 \begin{cases}
 0,&k<i,\\
 r-i,&k\ge i,
 \end{cases}
 \qquad
 k\text{-}\dell_{B_d}(L_i)=
 \begin{cases}
 0,&k\le i,\\
 r-i,&k>i.
 \end{cases}
\end{equation}
In particular, at the index $k=i<r$ the comparison has gap
$r-i$, and equality holds at every other index.
\end{corollary}
\begin{proof}
The relative syzygy chain is $E_1,E_2,\ldots,E_r$, with $E_r$
relative projective. If $k<i$, then $E_i=\Omega_{T_d}^kE_{i-k}$,
so its relative $k$-delooping level is zero. If $k\ge i$ and
$0\le t<r-i$, every nonprojective summand of a
$(t+k)$th relative syzygy has index at least $t+k+1>t+i$.
It cannot contain $\Omega_{T_d}^tE_i=E_{i+t}$ as a stable retract.
At $t=r-i$ this syzygy is relative projective.

Over $B_d$, the same chain has the additional initial object
$L_0$. Thus $L_i=\Omega_{B_d}^kL_{i-k}$ is available precisely
when $k\le i$. For $k>i$ and $t<r-i$, every nonprojective
$(t+k)$th syzygy has index at least $t+k>t+i$, giving the same
obstruction. The case $i=r$ gives zero on both sides for every $k$.
\end{proof}

The strict case also makes the equality criterion explicit.
For $1\le i<r$, the inclusion $L_i\to Q_{i-1}$ is a left
projective approximation. Thus $\Sigma_{B_d}^iL_i=L_0$ stably.
The first three projectives in the resolution of $L_0$ become
$P_2\to P_0\to S_0$ after tensoring with $T_d$, so
\[
 \Tor_1^{B_d}(L_0,T_d)=P_1/S_2\cong S_1\ne0.
\]
At the strict index $k=i$, the ordinary canonical witness is
therefore exactly the simple module omitted from $\Y_d$.
This identifies the obstruction detected by
\cref{prop:comparison-equality}.

\begin{example}[The three-vertex case and suspension]\label{ex:strict}
For $d=1$, the algebra $A_1$ is hereditary and
$T_1=S_0\oplus P_0\oplus S_2$ in the notation of
\cref{thm:unbounded-gap}. The only nonprojective
indecomposable of $\Fac T_1$ is $X_1=P_0/S_2$.
The quotient $X_1\to S_0$ is a left $\add T_1$-approximation:
$\Hom(X_1,S_0)=K$ and $\Hom(X_1,P_0)=\Hom(X_1,S_2)=0$.
It is not injective, so $X_1$ is not $T_1$-torsionless, and its
cokernel gives $\Sigma_{T_1}X_1=0$.

Over $B_1$, by contrast, $F_1(X_1)=L_1$ embeds into $Q_0$.
This embedding is a left projective approximation, so
$\Sigma_{B_1}L_1=L_0$ stably. Applying $G_1=-\otimes_{B_1}T_1$ to
$Q_1\to Q_0\to L_0\to0$ gives
$P_0\to S_0\to G_1(L_0)\to0$, whence $G_1(L_0)=0$.
Thus both the change of coefficient module in torsionfreeness
and the application of $G_1$ in relative suspension matter.
Here $\delta_{T_1}=2$, whereas the unshifted test terms have
maximum $1$.
\end{example}

\subsection{Tensor extensions with infinite global dimension}

The gap in \cref{thm:unbounded-gap} persists after adding
local coefficients. We first give the functorial reason.
All tensor products in this subsection are over $K$.

\begin{theorem}[Tensor induction]\label{thm:tensor-induction}
Let $H$ and $R$ be nonzero finite-dimensional basic $K$-algebras,
and let $U$ be a nonzero basic support $\tau$-tilting $H$-module.
Put $\widetilde H=R\otimes H$, $\widetilde U=R\otimes U$,
$B=\End_H(U)$ and $\mathsf I=R\otimes-$.
Then $\widetilde U$ is a basic support $\tau$-tilting
$\widetilde H$-module and
\begin{equation}\label{eq:tensor-endomorphism}
 \widetilde B:=\End_{\widetilde H}(\widetilde U)
 \cong R\otimes B.
\end{equation}
On relative stable categories, induction satisfies
\begin{equation}\label{eq:tensor-suspension}
 \Omega_{\widetilde U}\mathsf I\cong\mathsf I\Omega_U,
 \qquad
 \Sigma_{\widetilde U}\mathsf I\cong\mathsf I\Sigma_U.
\end{equation}
For $M\in\Fac U$ and $k\ge1$, one has, including infinite values,
\begin{equation}\label{eq:tensor-relative-level}
 \begin{gathered}
 k\text{-}\dell_{\widetilde U}(\mathsf I M)=k\text{-}\dell_U(M),
 \qquad \pd_{\widetilde U}(\mathsf I M)=\pd_U M,\\
 M\text{ is }k\text{-}U\text{-torsionfree}
 \quad\Longleftrightarrow\quad
 \mathsf I M\text{ is }k\text{-}\widetilde U\text{-torsionfree}.
 \end{gathered}
\end{equation}
Moreover, as left $\widetilde B$-modules up to projective summands,
\begin{equation}\label{eq:tensor-transpose}
 \Tr_{\widetilde U}(\mathsf I M)\cong R\otimes\Tr_U M.
\end{equation}
For every $V\in\modu H$, one also has
\begin{equation}\label{eq:tensor-ordinary-level}
 \begin{gathered}
 k\text{-}\dell_{\widetilde H}(R\otimes V)
 =k\text{-}\dell_H(V),\\
 \pd_{\widetilde H}(R\otimes V)=\pd_HV.
 \end{gathered}
\end{equation}
If $U$ is classical $1$-tilting, then so is $\widetilde U$.
\end{theorem}
\begin{proof}
The support $\tau$-tilting assertion is
\cite[Theorem~3.7]{LiZhangTensor}. The natural tensor--Hom
identifications give \eqref{eq:tensor-endomorphism}; see also
\cite[Lemma~2.6]{LiZhangTensor}. Since $K$ is algebraically
closed and $R,B$ are basic, the radical of $R\otimes B$ is
$\rad R\otimes B+R\otimes\rad B$, with quotient a product of
copies of $K$. Thus $R\otimes B$, and hence $\widetilde U$,
is basic.

Induction $\mathsf I$ and restriction $\mathsf E$
along $H\to R\otimes H$, $h\mapsto1\otimes h$, are exact.
They map $\Fac U$ into $\Fac\widetilde U$ and
$\Fac\widetilde U$ into $\Fac U$, respectively, and preserve
the indicated relative projectives. They therefore induce stable
functors commuting with relative syzygies, by applying them to
projective conflations and using the relative Schanuel lemma.
For suspension, tensor--Hom identifies maps between induced
modules with $R\otimes\Hom_H(-,-)$. Hence a left $\add U$-
approximation becomes a left $\add\widetilde U$-approximation:
this holds first for targets induced from $\add U$, and then
for their direct summands. Exactness of induction and the
cokernel model of \cref{thm:adjunction} give the second
isomorphism in \eqref{eq:tensor-suspension}, naturally in the
stable categories.

Induction sends every relative delooping retraction for $M$ to
one for $\mathsf I M$. Conversely, restriction sends a
retraction for $\mathsf I M$ at level $t$ to one for
$M^{\oplus h}$ at the same level, where $h=\dim_KR>0$.
Composing with the inclusion and projection of one copy of
$M$ gives a retraction for $M$ at level $t$. This proves the
delooping equality with no restriction on the original witness.
Applying the same functors to relative projective resolutions
gives equality of relative projective dimensions, since
$\pd_U(M^{\oplus h})=\pd_U M$.

To prove \eqref{eq:tensor-transpose}, induce a proper $\add U$-
presentation and use
\[
 \Hom_{\widetilde H}(\mathsf I U_i,\widetilde U)
 \cong R\otimes\Hom_H(U_i,U).
\]
Taking the cokernels defining the relative transposes proves
the assertion up to projective summands, which do not affect
positive Ext. Since the first tensor factor is the regular
left $R$-module, the tensor--Ext formula gives, for $j\ge1$,
\begin{equation}\label{eq:tensor-torsionfree-ext}
 \begin{split}
 \Ext^j_{\widetilde B^{\op}}
  (\Tr_{\widetilde U}(\mathsf I M),{}_{\widetilde B}\widetilde U)
 \cong R\otimes
 \Ext^j_{B^{\op}}(\Tr_U M,{}_B U)
 \end{split}
\end{equation}
as $K$-vector spaces; the other terms vanish because the regular
left $R$-module is projective. Since $R\ne0$, simultaneous
vanishing for $1\le j\le k$ is equivalent on the two sides.

Taking $U=H$ gives \eqref{eq:tensor-ordinary-level}.
Finally, if $U$ is classical $1$-tilting, inducing its
projective resolution and tilting coresolution, together with
$\Ext^1_{\widetilde H}(\widetilde U,\widetilde U)
\cong R\otimes\Ext^1_H(U,U)=0$, proves the last assertion.
\end{proof}

For ordinary delooping, \cite[Lemma~4.1(3)]{GuoSymmetry} gives the induction
inequality. Restriction supplies the converse for induced modules. The relative argument
above also controls witnesses anywhere in $\Fac\widetilde U$.
It does not assert preservation for arbitrary modules over
$R\otimes H$. \cref{thm:tensor-induction} also shows that
the torsionfreeness hypotheses on $\Sigma_U^n\Omega_U^nM$ in
\cref{thm:higher} are preserved and reflected by induction.
Locality of $R$ is needed for the dimension formulas below,
but not for these compatibility statements.

\begin{proposition}[Local coefficients over finite global dimension]\label{prop:tensor-finite-global}
Let $H$ be a finite-dimensional basic $K$-algebra of finite
global dimension $g$, and let $R$ be as in
\cref{thm:tensor-induction}. Assume in addition that $R$
is local with $R/\rad R=K$. For $\Lambda=R\otimes H$,
\begin{equation}\label{eq:tensor-dimensions}
 \begin{aligned}
 \findim\Lambda=\Findim\Lambda=\ddell\Lambda=\dell\Lambda&=g,\\
 \findim\Lambda^{\op}=\Findim\Lambda^{\op}
 =\ddell\Lambda^{\op}=\dell\Lambda^{\op}&=g.
 \end{aligned}
\end{equation}
If $\rad R\ne0$, every simple left or right $\Lambda$-module
has infinite projective dimension.
\end{proposition}
\begin{proof}
Every simple right $\Lambda$-module has the form
$(R/\rad R)\otimes S$ with $S$ simple over $H$.
Choose an embedding $R/\rad R\hookrightarrow R$ and tensor it
with $S$. \cref{thm:tensor-induction} gives an embedding into
$R\otimes S$ of projective dimension at most $g$.
\cref{lem:finite-pd-embedding} yields $\dell\Lambda\le g$.
On the other hand, inducing an $H$-module of projective
dimension $g$ gives $\findim\Lambda\ge g$.
Apply the same arguments to $R^{\op}$ and $H^{\op}$, whose
global dimension is also $g$. The chain
\eqref{eq:dimension-chain} on both sides proves all equalities.

For the last assertion, restriction to $R$ preserves projectives,
while a simple $\Lambda$-module restricts to copies of
$R/\rad R$. This residue module has infinite projective
dimension when $R$ is not semisimple. Indeed, if $R$ has Loewy
length $\ell$, a positive syzygy in a minimal resolution has
Loewy length at most $\ell-1$, whereas every nonzero finitely
generated projective $R$-module is free and has Loewy length
$\ell$. Such a resolution cannot terminate. The proof on the
left is identical.
\end{proof}

The little finitistic-dimension equalities also follow from
\cite[Theorem~1.5]{XiXu}, since $\findim R=\findim R^{\op}=0$.
The embedding argument identifies the ordinary delooping levels;
the big finitistic dimensions and derived delooping levels then
follow from \eqref{eq:dimension-chain}.

\begin{theorem}[Unbounded gaps in infinite global dimension]\label{thm:tensor-gap}
Let $d\ge1$ and let $R$ be a nonsemisimple finite-dimensional
local $K$-algebra with residue field $K$. With the notation of
\cref{thm:unbounded-gap}, put
\[
 \begin{gathered}
 \widehat A_d=R\otimes A_d,\qquad
 \widehat T_d=R\otimes T_d,\qquad
 \widehat X_d=R\otimes X_d,\\
 \widehat B_d=\End_{\widehat A_d}(\widehat T_d)
 \cong R\otimes B_d,\qquad
 \widehat F_d=\Hom_{\widehat A_d}(\widehat T_d,-).
 \end{gathered}
\]
Then $\widehat T_d$ is classical tilting of projective dimension
one, and both $\widehat A_d$ and $\widehat B_d$ have infinite
global dimension. Every simple module on either side of either
algebra has infinite projective dimension. Nevertheless,
\begin{equation}\label{eq:tensor-gap-dimensions}
 \begin{gathered}
 \findim\widehat A_d=\findim\widehat A_d^{\op}=d,\\
 \findim\widehat B_d=\findim\widehat B_d^{\op}=d+1.
 \end{gathered}
\end{equation}
The corresponding big finitistic dimensions and both ordinary
and derived delooping levels have the same respective values.
For the indicated module,
\begin{equation}\label{eq:tensor-gap}
 \begin{gathered}
 k\text{-}\dell_{\widehat T_d}(\widehat X_d)=d\quad(k\ge1),\\
 \dell_{\widehat B_d}\widehat F_d(\widehat X_d)=0,
 \qquad
 k\text{-}\dell_{\widehat B_d}\widehat F_d(\widehat X_d)=d
 \quad(k\ge2).
 \end{gathered}
\end{equation}
\end{theorem}
\begin{proof}
\cref{thm:tensor-induction} gives the tilting assertion and
identifies $\widehat F_d(\widehat X_d)$ with
$R\otimes F_d(X_d)$. Apply its relative and ordinary equalities
to \eqref{eq:gap-all-levels}. Apply
\cref{prop:tensor-finite-global} to $H=A_d$ and
$H=B_d$ to obtain the dimension formulas and the assertions
about simple modules.
\end{proof}

\begin{corollary}[Sharpness at every higher index]\label{cor:all-index-gaps}
For every pair of integers $k,d\ge1$, there exist a
finite-dimensional algebra $A$, a classical tilting $A$-module
$T$ with $\pd_AT=1$, and $M\in\Fac T$ such that, for
$B=\End_A(T)$ and $F=\Hom_A(T,-)$,
\begin{equation}\label{eq:all-index-gaps}
 k\text{-}\dell_BF(M)=0<d=k\text{-}\dell_T(M)
 =(k+1)\text{-}\dell_BF(M).
\end{equation}
Both $A$ and $B$ can be required to have infinite global
dimension, with every simple module on either side of infinite
projective dimension, while
\[
 \findim A=\findim A^{\op}=d+k-1,
 \qquad \findim B=\findim B^{\op}=d+k.
\]
The big finitistic dimensions and the ordinary and derived
delooping levels have the same respective values.
\end{corollary}
\begin{proof}
Put $n=d+k-1$ and use the family $A_n,T_n,B_n$.
Its last vertex is $r=n+1=d+k$.
Take $E_k=X_n$ if $k=1$, and $E_k=S_k$ if $k\ge2$.
\cref{cor:gap-profile}, with $i=k$, gives
\eqref{eq:all-index-gaps} before induction, since $r-k=d$.
Choose any nonsemisimple finite-dimensional local algebra $R$
with residue field $K$ and set
\[
 A=R\otimes A_n,\qquad T=R\otimes T_n,
 \qquad M=R\otimes E_k.
\]
\cref{thm:tensor-induction} preserves the three levels.
The remaining assertions follow from
\cref{thm:tensor-gap} with parameter $n$.
\end{proof}

Consequently, for any fixed $k$, there is no universal additive
constant comparing $k$-$\dell_T$ with $k$-$\dell_BF$, even when
the projective dimension of the tilting module is fixed at one.

For example, take $R=K[u,v]/(u,v)^3$ and $d=4$. Then
$\dim_K\widehat A_4=72$ and $\dim_K\widehat B_4=66$.
Their left and right finitistic dimensions are $4$ and $5$,
respectively, and the $12$-dimensional module $\widehat X_4$
has relative delooping level $4$ while its image has ordinary
delooping level zero. On the same algebra, $R\otimes S_i$
realizes comparison gap $5-i$ at index $k=i$ for $2\le i\le4$.
In general,
$\dim_K\widehat A_d=(2d+4)\dim_KR$ and
$\dim_K\widehat B_d=(2d+3)\dim_KR$.

\section{Cyclic algebras with local coefficients}\label{sec:cyclic}

In this section, we prove \cref{intro:cyclic} and calculate two explicit families. The construction allows arbitrary local coefficients and independent automorphism twists on the arrows. Lifting path modules gives both the upper bound from simple-module embeddings and modules attaining that bound.

\subsection{Construction and a uniform dimension formula}\label{subsec:cyclic-uniform}

Let $R$ be a finite-dimensional local $K$-algebra with
$R/\rad R\cong K$. No commutativity or self-injectivity assumption is made
on $R$. Let $Q$ be the oriented cycle with vertices
$\mathbb Z/n\mathbb Z$, where $n\ge2$, and arrows $a_i:i\to i+1$.
Fix a tuple $c=(c_0,\ldots,c_{n-1})$ satisfying
\begin{equation}\label{eq:cyclic-admissibility}
 c_i\ge2,\qquad c_i\le c_{i+1}+1,
\end{equation}
and choose $K$-algebra automorphisms
$\boldsymbol{\sigma}=(\sigma_0,\ldots,\sigma_{n-1})$ of $R$.
Put a copy $R_i$ of $R$ at each vertex, and write $r_i$ for the image
of $r\in R$ in $R_i$. Define $\Lambda=\Lambda_R(c;\boldsymbol{\sigma})$
by the coefficient multiplications in the $R_i$, the crossing relations
\begin{equation}\label{eq:cyclic-crossing}
 a_i r_{i+1}=\sigma_i(r)_i a_i\qquad(r\in R),
\end{equation}
and the truncation relations
\begin{equation}\label{eq:cyclic-truncation}
 p_{i,c_i}=0,\qquad
 p_{i,\ell}=a_i a_{i+1}\cdots a_{i+\ell-1},\qquad p_{i,0}=e_i.
\end{equation}
Lengths count cycle arrows only. The underlying Nakayama algebra is
\begin{equation}\label{eq:underlying-nakayama}
 N=N(c)=KQ/(p_{i,c_i}\mid i\in\mathbb Z/n\mathbb Z).
\end{equation}

Here is a normal form for this presentation. Set
$\sigma_{i,\ell}=\sigma_i\circ\cdots\circ\sigma_{i+\ell-1}$, with
$\sigma_{i,0}=1$. Before truncation, multiplication is given by
\[
 (r_i p_{i,\ell})(s_{i+\ell}p_{i+\ell,t})
   =(r\sigma_{i,\ell}(s))_i p_{i,\ell+t},
\]
and products with nonmatching endpoints vanish. This multiplication is
associative. The span of the terms with $\ell\ge c_i$ is a two-sided
ideal: closure under prepending an arrow follows from
$c_{i-1}\le c_i+1$, and closure under appending is immediate. Thus, if
$\mathcal B$ is a $K$-basis of $R$, a basis of $\Lambda$ is
\begin{equation}\label{eq:cyclic-basis}
 b_i p_{i,\ell},\qquad b\in\mathcal B,\quad 0\le\ell<c_i.
\end{equation}
In particular,
\begin{equation}\label{eq:cyclic-dimension}
 \dim_K\Lambda=(\dim_KR)\sum_i c_i.
\end{equation}
The ideal generated by the $\rad R_i$ and the cycle arrows is nilpotent,
since each $\sigma_i$ preserves $\rad R$. Its quotient is $K^n$,
so it is $\rad\Lambda$. Denote the simple right $\Lambda$-modules by
$S_i$ and the simple right $N$-modules by $U_i$.

Write $P_i=e_i\Lambda$, let $\mathfrak a$ be the ideal generated by the
cycle arrows, and put
\[
 M(i,\ell)=P_i/P_i\mathfrak a^\ell,\qquad
 V(i,\ell)=e_iN/e_i(\rad N)^\ell
 \quad(1\le\ell\le c_i).
\]
The modules $V(i,\ell)$ are all the indecomposable right $N$-modules.
The module $M(i,\ell)$ has simple top $S_i$, but its coefficient
layers need not be simple. In particular, $W_i=M(i,1)$ is the regular
$R$-module supported at vertex $i$.

\begin{lemma}[Lifting path modules]\label{lem:cyclic-syzygies}
For $1\le\ell<c_i$,
\begin{equation}\label{eq:cyclic-syzygy}
 \Omega_\Lambda M(i,\ell)\cong M(i+\ell,c_i-\ell),
 \qquad M(i,c_i)=P_i.
\end{equation}
Consequently, including infinite values,
\begin{equation}\label{eq:cyclic-pd-lifting}
 \pd_\Lambda M(i,\ell)=\pd_N V(i,\ell).
\end{equation}
Moreover, $S_{i+\ell-1}$ embeds in $M(i,\ell)$.
\end{lemma}
\begin{proof}
The kernel of $P_i\to M(i,\ell)$ is $p_{i,\ell}\Lambda$. Left
multiplication by $p_{i,\ell}$ defines a surjection
$P_{i+\ell}\to p_{i,\ell}\Lambda$. By the normal form
\eqref{eq:cyclic-basis}, its kernel is
$P_{i+\ell}\mathfrak a^{c_i-\ell}$. Here
$c_i-\ell\le c_{i+\ell}$ follows from
\eqref{eq:cyclic-admissibility}. Invertibility of the coefficient
automorphisms ensures that no extra kernel occurs. Since these kernels
are contained in the radicals of the projectives, the maps are
projective covers. The same calculation with $R=K$ gives the identical
recursion for $V(i,\ell)$. A module $M(i,\ell)$ is projective exactly
when $\ell=c_i$: its projective cover is $P_i$, and for $\ell<c_i$ this
cover has a nonzero kernel. Thus the two minimal resolutions terminate
at exactly the same step, proving \eqref{eq:cyclic-pd-lifting}.

Choose $0\ne z\in R$ with $z\rad R=0$. The line
$Kz_i p_{i,\ell-1}$ in $M(i,\ell)$ is supported at vertex
$j=i+\ell-1$. It is annihilated by all cycle arrows and by
$\rad R_j$, because $\sigma_{i,\ell-1}(\rad R)=\rad R$.
Since $R/\rad R=K$, this line is a submodule isomorphic to $S_j$.
\end{proof}

In particular, $S_i$ embeds in $W_i$. Whenever all the $W_i$ have
finite projective dimension, \cref{lem:finite-pd-embedding} gives
\begin{equation}\label{eq:cyclic-upper}
 \findim\Lambda^{\op}\le\dell\Lambda
 \le\max_i\pd_\Lambda W_i.
\end{equation}
The next theorem also covers the case where some $W_i$ have infinite
projective dimension.

\begin{theorem}[Local coefficients and automorphism twists]\label{thm:cyclic-uniform}
For every $R,c,\boldsymbol{\sigma}$ as above, put
$\Lambda=\Lambda_R(c;\boldsymbol{\sigma})$. Then
\begin{equation}\label{eq:cyclic-uniform}
 \findim\Lambda=\findim\Lambda^{\op}
 =\dell\Lambda=\dell\Lambda^{\op}=\findim N(c)<\infty.
\end{equation}
Thus these four invariants depend only on the underlying Nakayama
algebra. There is no assumption that $N(c)$ has finite global dimension.
\end{theorem}
\begin{proof}
For each $i$, let $E_i$ be the injective envelope of $U_i$ in
$\modu N$, and set
\[
 e_i'=\min\{\pd_N U_i,\pd_N E_i\}.
\]
Ringel's formula for Nakayama algebras \cite{RingelNakayama} states that
every $e_i'$ is finite and that
\begin{equation}\label{eq:ringel-input}
 f:=\findim N=\findim N^{\op}=\max_i e_i'.
\end{equation}
Choose $H_i\in\{U_i,E_i\}$ with $\pd_NH_i=e_i'$. It is
indecomposable with socle $U_i$, so $H_i=V(t_i,\ell_i)$ for suitable
$t_i,\ell_i$ with $t_i+\ell_i-1=i$ in $\mathbb Z/n\mathbb Z$.
By \cref{lem:cyclic-syzygies}, the module
$\widetilde H_i=M(t_i,\ell_i)$ satisfies
\[
 S_i\lhook\joinrel\longrightarrow\widetilde H_i,
 \qquad \pd_\Lambda\widetilde H_i=e_i'.
 \]
\cref{lem:finite-pd-embedding} and an index attaining the maximum
in \eqref{eq:ringel-input} therefore give
\begin{equation}\label{eq:cyclic-first-side}
 \findim\Lambda^{\op}\le\dell\Lambda\le f
 \quad\hbox{and}\quad
 \findim\Lambda\ge f.
\end{equation}

We verify the opposite-algebra step explicitly. Reverse the cycle.
At a vertex $j$, the surviving paths ending at $j$ have consecutive
lengths $0,\ldots,d_j-1$, where
\[
 d_j=\#\{\ell\ge0\mid \ell<c_{j-\ell}\}.
\]
Indeed, every suffix of a surviving path survives. In $\Lambda^{\op}$
these become the paths starting at $j$, and the coefficient algebra
is $R^{\op}$. For the reversed arrow $b_i=a_i^{\op}:i+1\to i$,
\eqref{eq:cyclic-crossing} becomes
\[
 b_i r_i^{\op}
   =\sigma_i^{-1}(r)_{i+1}^{\op}b_i.
\]
After relabeling the vertices, this is the same construction with
lengths $d_j$ and inverse automorphisms. Its underlying Nakayama
algebra is $N^{\op}$. The argument proving
\eqref{eq:cyclic-first-side} thus applies to $\Lambda^{\op}$.
By \eqref{eq:ringel-input}, it gives
\[
 \findim\Lambda\le\dell\Lambda^{\op}\le f
 \quad\hbox{and}\quad
 \findim\Lambda^{\op}\ge f.
\]
Combining the two pairs of inequalities proves \eqref{eq:cyclic-uniform}.
\end{proof}

\begin{corollary}\label{cor:cyclic-big-derived}
Under the hypotheses of \cref{thm:cyclic-uniform},
\[
 \Findim\Lambda=\Findim\Lambda^{\op}
 =\ddell\Lambda=\ddell\Lambda^{\op}=\findim N(c).
\]
\end{corollary}
\begin{proof}
Write $f=\findim N(c)$. The theorem and
\eqref{eq:dimension-chain} give
\[
 f=\findim\Lambda^{\op}\le\Findim\Lambda^{\op}
 \le\ddell\Lambda\le\dell\Lambda=f.
\]
Apply the same argument to $\Lambda^{\op}$.
\end{proof}

\begin{remark}[Relation to twisted tensor products]\label{rem:cyclic-xixu}
The finitistic-dimension equalities in
\cref{thm:cyclic-uniform} can also be deduced from
\cite[Theorem~1.5]{XiXu}. Indeed, $\Lambda$ contains
$C_0=\prod_iR_i$, the scalar path subalgebra $N$, and their common
maximal semisimple subalgebra $S=\prod_iKe_i$.
The normal form identifies multiplication with an isomorphism
$C_0\otimes_S N\cong\Lambda$ of $C_0$--$N$-bimodules, and the
crossing relations give
$\rad N\,\rad C_0=\rad C_0\,\rad N$.
Thus $\Lambda$ is a twisted tensor product in the sense of Xi--Xu.
A finite-dimensional local algebra has finitistic dimension zero
on both sides, so $\findim C_0=0$. Xi--Xu's inequalities, applied
on each side, give $\findim\Lambda=\findim N$ and
$\findim\Lambda^{\op}=\findim N^{\op}$.
The proof above also computes both delooping levels and supplies
the embeddings and finite-projective-dimension modules realizing
the common value.
\end{remark}

\begin{corollary}\label{cor:cyclic-simple-infinite}
If $\rad R\ne0$, every simple left or right
$\Lambda_R(c;\boldsymbol{\sigma})$-module has infinite projective
dimension. In particular, its global dimension is infinite.
\end{corollary}
\begin{proof}
Embed $R$ diagonally by $r\mapsto\sum_i r_i$. The normal form
\eqref{eq:cyclic-basis} shows that restriction sends every left or
right $\Lambda$-projective to a projective $R$-module. Each path
component is a regular $R$-module, possibly with an automorphism
twist. A simple $\Lambda$-module restricts to $R/\rad R$.

The residue module has infinite projective dimension on both
sides by the Loewy-length argument in the proof of
\cref{prop:tensor-finite-global}. Restriction therefore
proves the assertion.
\end{proof}

\begin{example}[An underlying algebra of infinite global dimension]\label{ex:cyclic-infinite-base}
Take $c=(3,3,4)$ and $n=3$. For $N=N(c)$ the simple $U_0$ satisfies
$\Omega_N^2U_0\cong U_0$, so $\gldim N=\infty$. The injective
envelopes of $U_0,U_1,U_2$ are respectively
$V(1,3),V(2,3),V(2,4)$, with projective dimensions $0,2,0$.
Since $\Omega_N^3U_1\cong U_0$ and $\pd_NU_2=1$,
\eqref{eq:ringel-input} gives $\findim N=2$.
\cref{thm:cyclic-uniform} yields
\[
 \findim\Lambda_R((3,3,4);\boldsymbol{\sigma})
 =\findim\Lambda_R((3,3,4);\boldsymbol{\sigma})^{\op}=2
\]
for every choice of the automorphisms. The same value holds for the
two delooping levels. The lifted module $M(2,3)$ has a minimal
resolution
\[
 0\longrightarrow P_0\xrightarrow{a_2}P_2
 \xrightarrow{p_{2,3}}P_2\longrightarrow M(2,3)\longrightarrow0.
\]
For $R=K[u,v]/(u,v)^3$, this gives a $60$-dimensional algebra.
Here $\pd_\Lambda W_0=\pd_\Lambda W_1=\infty$, so the bound
\eqref{eq:cyclic-upper} using only the vertex modules is insufficient;
the lifted injective envelopes provide the required embeddings.
\end{example}

\subsection{Two explicit families with prescribed finitistic dimension}\label{subsec:cyclic-families}

Fix $m\ge3$ and $q\in K^\times$, and put
\[
 R_m=K[u,v]/(u,v)^m,\qquad
 h_m=\dim_KR_m=\frac{m(m+1)}2.
\]
We use the notation
\[
 \begin{gathered}
 \Lambda_m(c;q)=\Lambda_{R_m}(c;\boldsymbol{\sigma}),
 \qquad
 \sigma_i=1\ (0\le i<n-1),\\
 \sigma_{n-1}(u)=qu,\quad \sigma_{n-1}(v)=v.
 \end{gathered}
\]
At vertex $i$, write $u_i,v_i$ for the two coefficient generators.
In \eqref{eq:cyclic-basis} one may take
$\mathcal B=\{u^rv^s\mid r,s\ge0,\ r+s<m\}$.
The following two families are consequences of
\cref{thm:cyclic-uniform}; the computations also give modules
attaining their finitistic dimensions.

\begin{example}[Even finitistic dimensions]\label{ex:cyclic-even}
For $n\ge3$ set
\[
 B_{n,m}^+(q)=\Lambda_m\bigl((n,n+1,\ldots,n+1);q\bigr).
\]
Then
\begin{equation}\label{eq:cyclic-even-result}
 \begin{gathered}
 \dim_KB_{n,m}^+(q)=h_m(n^2+n-1),\\
 \findim B_{n,m}^+(q)=\findim\bigl(B_{n,m}^+(q)\bigr)^{\op}=2n-2,\\
 \dell B_{n,m}^+(q)=\dell\bigl(B_{n,m}^+(q)\bigr)^{\op}=2n-2.
 \end{gathered}
\end{equation}
Indeed, \eqref{eq:cyclic-syzygy} gives
\[
 \Omega^{2r}W_0\cong M(0,r+1)\quad(0\le r\le n-1),
\]
whose last term is $P_0$. Also,
\[
 \Omega W_{n-1}\cong P_0,\qquad
 \Omega^2W_i\cong W_{i+1}\quad(1\le i\le n-2).
\]
All intermediate syzygies before the indicated projective terms are nonprojective. Therefore
\begin{equation}\label{eq:cyclic-even-pd}
 \pd W_0=2n-2,\qquad
 \pd W_i=2n-2i-1\quad(1\le i\le n-1).
\end{equation}
By \eqref{eq:cyclic-pd-lifting}, these are also the projective dimensions of the simple modules over $N(n,n+1,\ldots,n+1)$. Hence that Nakayama algebra has global dimension $2n-2$. \cref{thm:cyclic-uniform} proves the four homological equalities in \eqref{eq:cyclic-even-result}; the dimension formula follows from \eqref{eq:cyclic-dimension}.
\end{example}

\begin{example}[Odd finitistic dimensions]\label{ex:cyclic-odd}
For $n\ge4$ set
\[
 B_{n,m}^-(q)=\Lambda_m\bigl((n-1,\ldots,n-1,n);q\bigr).
\]
Then
\begin{equation}\label{eq:cyclic-odd-result}
 \begin{gathered}
 \dim_KB_{n,m}^-(q)=h_m(n^2-n+1),\\
 \findim B_{n,m}^-(q)=\findim\bigl(B_{n,m}^-(q)\bigr)^{\op}=2n-3,\\
 \dell B_{n,m}^-(q)=\dell\bigl(B_{n,m}^-(q)\bigr)^{\op}=2n-3.
 \end{gathered}
\end{equation}
In this case \eqref{eq:cyclic-syzygy} gives
\[
 \Omega W_{n-1}\cong P_0,\qquad
 \Omega^2W_i\cong W_{i-1}\quad(0\le i\le n-3),
\]
where $W_{-1}=W_{n-1}$. For the remaining vertex one has
\[
 \Omega^{2r}W_{n-2}\cong M(n-2-r,r+1)
 \quad(0\le r\le n-2),
\]
ending in $M(0,n-1)=P_0$. It follows that
\begin{equation}\label{eq:cyclic-odd-pd}
 \begin{gathered}
 \pd W_i=2i+3\quad(0\le i\le n-3),\\
 \pd W_{n-2}=2n-4,\qquad \pd W_{n-1}=1.
 \end{gathered}
\end{equation}
The maximum is $2n-3$, attained by $W_{n-3}$. By \eqref{eq:cyclic-pd-lifting},
\[
 \gldim N(n-1,\ldots,n-1,n)=2n-3.
\]
\cref{thm:cyclic-uniform} and \eqref{eq:cyclic-dimension} prove \eqref{eq:cyclic-odd-result}.
\end{example}

\begin{proposition}\label{prop:cyclic-structure}
Every simple left or right module over either family in \cref{ex:cyclic-even,ex:cyclic-odd} has infinite projective dimension. Neither family is standardly stratified for any ordering of the simple modules.
\end{proposition}
\begin{proof}
The assertion about simple modules follows from \cref{cor:cyclic-simple-infinite}.

For the stratification assertion, put $I_j=\Lambda e_j\Lambda$. If $\Lambda$ were standardly stratified, the ideal $I_j$ for a maximal vertex would be projective as a right $\Lambda$-module; see \cite[Section~1.2]{CruzMarczinzik}. Indeed, this trace ideal is filtered by the maximal standard module $e_j\Lambda$, so the filtration splits. We show that no vertex has this property. Let $t$ be the least nonnegative residue of $j-i$ modulo $n$. If $t<c_i$, then the shortest path from $i$ to $j$ gives
\begin{equation}\label{eq:cyclic-trace}
 e_iI_j=p_{i,t}\Lambda\cong M(j,c_i-t).
\end{equation}
Whenever $0<c_i-t<c_j$, this is a nonprojective direct summand of the right module $I_j$.

For $B_{n,m}^+(q)$ and $j>0$, choose $i=0$. Then $t=j$ and $c_i-t=n-j<c_j=n+1$. For $j=0$, choose $i=1$; then $c_i-t=2<c_0=n$. For $B_{n,m}^-(q)$ and $j>0$, choose $i=j-1$. This gives $c_i-t=n-2<c_j$. For $j=0$, choose $i=2$; since $n\ge4$, one has $c_i-t=1<c_0=n-1$. Thus all $I_j$ are nonprojective, proving the claim.
\end{proof}

For $K=\mathbb C$, $m=3$ and $q=2$, write $B_n^\pm=B_{n,3}^\pm(2)$. Some explicit values are given in \cref{tab:cyclic-examples}. Both families have infinite global dimension by \cref{prop:cyclic-structure}, while together they realize every integer $d\ge4$ as their left and right finitistic dimension.

\begin{table}[htbp]
\centering
\caption{Cyclic examples with $R_3=\mathbb C[u,v]/(u,v)^3$ and $q=2$.}\label{tab:cyclic-examples}
\begin{tabular}{@{}cccc@{}}
\toprule
Algebra & Cycle-path truncation lengths & $\dim_KB$ & $\findim B=\findim B^{\op}$\\
\midrule
$B_3^+$ & $(3,4,4)$ & $66$ & $4$\\
$B_4^-$ & $(3,3,3,4)$ & $78$ & $5$\\
$B_4^+$ & $(4,5,5,5)$ & $114$ & $6$\\
$B_5^-$ & $(4,4,4,4,5)$ & $126$ & $7$\\
$B_5^+$ & $(5,6,6,6,6)$ & $174$ & $8$\\
\bottomrule
\end{tabular}
\end{table}

For instance, the cycle relations of $B_3^+$ are generated by $a_0a_1a_2=0$ and $a_1a_2a_0a_1=0$. A minimal projective resolution of $W_0$ is
\begin{equation}\label{eq:cyclic-66-resolution}
 0\longrightarrow P_0\xrightarrow{a_2}P_2
 \xrightarrow{a_0a_1}P_0\xrightarrow{a_1a_2}P_1
 \xrightarrow{a_0}P_0\longrightarrow W_0\longrightarrow0,
\end{equation}
where the labeled maps are left multiplication by the indicated paths. This displays a module of projective dimension four in the $66$-dimensional algebra. Varying $m$ gives infinitely many different algebra dimensions for each fixed finitistic dimension in either family.

\begin{remark}\label{rem:cyclic-scope}
These computations apply the specialization $A=T=\Lambda$ of the article, so $\C=\modu\Lambda$, $\mathcal J$ is empty, and $\delta_T=\dell\Lambda$. In the two explicit families, the modules $W_i$ supply finite-projective-dimension embeddings for all simple modules. For general $c$, \cref{thm:cyclic-uniform} also uses lifts of injective Nakayama modules, as illustrated by \cref{ex:cyclic-infinite-base}. If $q=1$, the construction is the ordinary tensor product
\[
 \Lambda_m(c;1)\cong R_m\otimes_K
 \bigl(KQ/(p_{i,c_i}\mid0\le i<n)\bigr),
\]
where $Q$ is the oriented cycle without the coefficient loops. The uniform theorem and \cref{cor:cyclic-big-derived} show that the ordinary and derived delooping levels and both finitistic dimensions are independent of all coefficient automorphisms. \cref{rem:cyclic-xixu} explains the existing finitistic-dimension result; the additional calculations identify the delooping levels and exhibit modules of maximal finite projective dimension. The complementary applications with nonprojective tilting modules are \cref{thm:unbounded-gap,thm:tensor-gap}.
\end{remark}

\par\medskip
\noindent\textbf{Acknowledgements.}\quad
The authors are supported by the National Natural Science Foundation of China (Nos.~12171207 and 12371038).

\par\medskip
\noindent\textbf{Data availability.}\quad
All mathematical constructions and proofs supporting the results are
contained in the article. No external datasets are required.

\par\medskip
\noindent\textbf{Use of generative AI.}\quad
ChatGPT (OpenAI) assisted with language editing, literature searches,
and the exploration and drafting of examples during
manuscript preparation. Responsibility for the final mathematical content
rests with the authors.

\begingroup
\hypersetup{urlcolor=black}

\endgroup

\medskip
{\footnotesize
\noindent\textbf{Mingfei Xu}\\
School of Mathematics and Statistics, Jiangsu Normal University,\\
Xuzhou 221116, Jiangsu, P.~R.~China\\
E-mail: \href{mailto:mfxu123@163.com}{mfxu123@163.com}

\vspace{5mm}
\noindent\textbf{Xiaojin Zhang}\\
School of Mathematics and Statistics, Jiangsu Normal University,\\
Xuzhou 221116, Jiangsu, P.~R.~China\\
E-mail: \href{mailto:xjzhang@jsnu.edu.cn}{xjzhang@jsnu.edu.cn}
\par}

\end{document}